\documentclass[11pt]{amsart}
\usepackage[colorlinks=blue,linkcolor=blue]{hyperref}
\usepackage{enumerate}
\usepackage{lineno}
\usepackage{graphicx}
\usepackage{geometry}
\usepackage{enumitem}
\makeatletter
\@namedef{subjclassname@}{%
\textup{}}
\makeatother
\newtheorem{thm}{Theorem}
\newtheorem{thmx}{Theorem}
\newtheorem{cor}[thmx]{Corollary}

\newtheorem{case}{Case}
\newtheorem{subcase}{Subcase}

\renewcommand{\thethmx}{\Alph{thmx}}
\newtheorem{rem}{Remark}
\newtheorem{Que}{Question}
\newtheorem*{define}{Definition}

\numberwithin{equation}{section}
 \newtheorem*{conj}{Hayman's Conjecture}

\begin{document}

\title[]{ Some new findings concerning value distribution of  a pair  of delay-differential polynomials }

\author[
]{ Jianren Long*, Xuxu Xiang}

\address{Jianren Long \newline School of Mathematical Sciences, Guizhou Normal University, Guiyang, 550025, P.R. China. }
\email{longjianren2004@163.com}

\address{Xuxu Xiang \newline School of Mathematical Sciences, Guizhou Normal University, Guiyang, 550025, P.R. China. }
\email{1245410002@qq.com}


\date{}


\begin{abstract}The paired Hayman's conjecture of different types
are considered. More accurately speaking,  the zeros  of a pair
of  $f^nL(z,g)-a_1(z)$ and  $g^mL(z,f)-a_2(z)$ are  characterized  using  different methods from those previously employed, where
$f$ and $g$ are both transcendental entire functions, $L(z,f)$ and  $L(z,g)$ are non-zero linear delay-differential polynomials, $\min\{n,m\}\ge 2$, $a_1,a_2$ are  non-zero small functions with relative to $f$ and $g$, or to $f^n(z)L(z,g)$ and $g^m(z)L(z,f)$, respectively.
These results give  answers  to three open questions raised by  Gao, Liu[Bull. Korean Math. Soc. 59 (2022)] and Liu, Liu[J. Math. Anal. Appl. 543 (2025)].
\end{abstract}



\keywords{Nevanlinna theory; Picard exceptional values; Entire functions; Paired Hayman's conjecture; Delay-differential polynomials.
This research work is supported by the National Natural Science Foundation of China (Grant No. 12261023, 11861023).\\
*Corresponding author. \\
}
\maketitle
\section{Introduction}

Let $f$ be a meromorphic function in the complex plane $\mathbb{C}$. Assume that the reader is familiar with the standard notation and basic results of Nevanlinna theory,  such as $m(r,f),~N(r,f)$,$~T(r,f)$, see \cite{hayman} for more details. A meromorphic
function $g$ is said to be a small function of $f$ if $T(r,g)=S(r,f)$, where $S(r,f)$  denotes any quantity
that satisfies $S(r,f)= o(T(r, f))$ as $r$ tends to infinity, outside a possible exceptional set of finite linear measure.
 $\rho(f)=\underset{r\rightarrow \infty}{\lim\sup}\frac{\log^+T(r,f)}{\log r}$ and   $\rho_2(f)=\underset{r\rightarrow \infty}{\lim\sup}\frac{\log^+\log^+T(r,f)}{\log r}$  are used to denote the order and the hyper-order  of $f$, respectively.  In general, the linear delay-differential polynomial of $f$ is defined by
\[L(z,f)=\sum_{i=1}^{m}b_i(z)[f^{(\nu_{i})}(z+c_{i})],\]
where~$m,\nu_{i}$~are nonnegative integers, coefficients~\(b_i(z)\)~are meromorphic functions.~If $c_i=0$ for $i=1,...,m$,~ then~\(L(z,f)\)~is the differential polynomial of $f$.~If~$\nu_{i}=0$ for$~i=1,~2,...,m$, then~$L(z,f)$~is the difference polynomial of $f$. The following definition is   the generalized Picard exceptional values or functions.

\begin{define}Let $f$ be a meromorphic function.
	If $f - \alpha$ has finitely many zeros, then  meromorphic function $\alpha$ is called a generalized Picard exceptional function of $f$.  Furthermore, if  $\alpha$ is also a small function of $f$,  then we say $f$  has a generalized Picard exceptional small function $\alpha$. If $\alpha$ is a finite constant, then we say that $f$  has a generalized Picard exceptional value $\alpha$.
\end{define}


The well-known  Picard theorem can be derived from the Nevanlinna's second fundamental theorem.
Here, we present the second fundamental theorem concerning three small functions\cite[Theorem 2.5]{hayman}: Let $f$ be a transcendental  meromorphic function, and let $a_j$ $(j=1,2,3)$ be small functions of $f$, then
\[T(r,f)<{\sum_{j=1}^3} \overline{N}(r,\frac{1}{f-a_j})+S(r,f).\]
The generalizations of Nevanlinna's second main theorem  concerning $q~(\ge3)$  small functions was given by Yamanoi\cite{ya}, where $q$ is an integer. It is also easy to see  a transcendental  meromorphic function has at most two generalized Picard exceptional small functions.

The study of the generalized Picard exceptional values of complex differential polynomials has a long and rich history of research, including notable contributions such as Milloux's inequality\cite[Theorem 3.2]{hayman}, Hayman's conjecture\cite{Hayman59}, Wiman's conjecture\cite{Wi}.
In 1959, Hayman considered the value distribution of complex differential polynomials in his significant paper\cite{Hayman59}. One of his results can be stated as follows.

\begin{renewcommand}{\thethm}{\Alph{thm}}
	\begin{thm}\cite[Theorem 10]{Hayman59}
		\label{thA}
		If $f$ is a transcendental entire function and $n~(\ge2)$ is a positive integer, then $f^nf'$ cannot have any non-zero generalized Picard exceptional values.
	\end{thm}
\end{renewcommand}
Clunie\cite{Clunie} proved that Theorem \ref{thA} is also true for the case $ n=1$. The well-known Hayman's conjecture is also presented in the same paper\cite{Hayman59}.
\begin{conj}\cite{Hayman59}
	If $f$ is a transcendental meromorphic function and $n$ is a positive integer, then $f^nf'$ cannot have any non-zero generalized Picard exceptional values.
\end{conj}

Hayman's conjecture has been solved completely. Hayman\cite[Corollary to Theorem 9]{Hayman59} obtained the proof for the case where $n\ge3$. Mues\cite{Mues} provided the proof for the case where $n=2$. Finally, using the theory of normal families, Bergweiler and Eremenko\cite{Bergweiler}, Chen and Fang\cite{Chen}, and Zalcman\cite{Zalcman} proved the case for $n=1$, respectively.

In 2007, Laine and Yang \cite[Theorem 2]{Laine07} considered the generalized Picard exceptional values of complex difference polynomials, as outlined below, which can be viewed the  difference analogues of  Theorem \ref{thA}.

\begin{renewcommand}{\thethm}{\Alph{thm}}
	\begin{thm}\cite{Laine07}
		\label{thB}
		If $f$ is a transcendental entire function of finite order and $n~(\ge2)$ is a positive integer, then $f^nf(z+c)$ cannot have any non-zero generalized Picard exceptional values, where  $c$ is a non-zero constant.
	\end{thm}
\end{renewcommand}

For other results concerning the generalized Picard exceptional values of differences of meromorphic functions, please refer to \cite{Be, czx,lk, long}.
The  generalized Picard exceptional small functions  of delay-differential polynomials related to Hayman's conjecture was considered by
Liu, Liu and Zhou\cite{Liu14}, and they obtained the following result.
\begin{renewcommand}{\thethm}{\Alph{thm}}
	\begin{thm}\cite{Liu14}
		\label{thC}
		Let $f$ be a transcendental meromorphic function of hyper-order $\rho_2(f)<1$, and let $c$ be a finite constant. If $n\ge2k + 6$ or if
		$n\ge3$ and $f$ is a transcendental entire function, then $f^nf^{(k)}(z+c)$
	can not have any non-zero generalized Picard exceptional small function $a$ with respect to $f$.
	\end{thm}
\end{renewcommand}

Later, the condition $n\ge2k + 6$ in Theorem \ref{thC} has been weakened by Laine et al.\cite{Laine20} to $n>k+4$.

\section{ Hayman's conjecture  on  a pair  of delay-differential polynomials}

In 2022, Gao  and Liu\cite{Gao22} considered the  Hayman's conjecture  regarding  a pair  of delay-differential polynomials.

\begin{renewcommand}{\thethm}{\Alph{thm} }
	\begin{thm}\cite{Gao22}
		Let $f$ and $g$ be both transcendental meromorphic functions, $n$ and $k$ be non-negative integers, $\alpha$ be a  non-zero small function with respect to $f$ and $g$. If one of the following conditions is satisfied:
		\label{thD}
		\begin{itemize}
			\item [(1)] $f$ and $g$ are both  entire functions, $n\ge3$, $\rho_2(f)<1$ and $\rho_2(g)<1$, $L(z,h)=h^{(k)}(z+c)$ or $L(z,h)=h(z+c)-h(z)$;
			\item [(2)]$f$ and $g$ are meromorphic functions,  $\rho_2(f)<1$ and $\rho_2(g)<1$, $L(z,h)=h(z+c)$, $n\ge 4$ or $L(z,h)=h(z+c)-h(z)$, $n\ge 5$ or $L(z,h)=h^{(k)}(z+c)$, $n\ge k+4$, respectively;
			\item [(3)]   $f$ and $g$ are entire functions, $n\ge3$, $L(z,h)=h^{(k)}(z)$, or $f$ and $g$ are meromorphic functions, $n\ge k+4$, $L(z,h)=h^{(k)}(z)$, respectively;
		\end{itemize}
		then  $f^nL(z,g)$ and $g^nL(z,f)$ cannot have a common non-zero generalized Picard exceptional small function $\alpha$.
	\end{thm}
\end{renewcommand}
Obviously, if $f\equiv g$, then Theorem \ref{thD} may reduce to Hayman's conjecture of different types.
Regarding Theorem D, Gao and Liu posed the following question for further investigation.
\begin{Que}\cite[Question 1]{Gao22}
	Can we reduce $n\ge3$ to $n\ge2$ in Theorem \ref{thD} for entire functions $f$ and $g$? And what is the sharp value $n$ for meromorphic
	functions $f,g$?
\end{Que}
By considering a broader class of a pair of delay-differential polynomials, we give a positive answer to the first part of Question 1.
\setcounter{thmx}{0}
\renewcommand{\thethmx}{\arabic{section}.\arabic{thmx}}
\begin{thmx}
	\label{th1.1}
	Let $f$ and $g$ be both transcendental entire functions, $\min \{n,m\}\ge 2$, $\alpha$ be a  non-zero small function with respect to $f$ and $g$. Then,  $f^nL(z,g)$ and $g^mL(z,f)$ cannot have a common non-zero generalized Picard exceptional small function $\alpha$, where $L(z,f)$ and  $L(z,g)$ are non-zero linear differential polynomial  with small entire functions of $f$ and $g$ as its coefficients.
\end{thmx}

\begin{rem}
	If  $\rho_2(f)<1$ and $\rho_2(g)<1$, $L(z,f)$ and  $L(z,g)$ are non-zero linear delay-differential  polynomial  with small entire functions of $f$ and $g$ as their coefficients, $\alpha$ is a non-zero small function  with respect to $f$ and $g$, then $f^nL(z,g)$ and $g^nL(z,f)$ cannot have a common non-zero generalized Picard exceptional small function $\alpha$  by using the similar proof of Theorem \ref{th1.1}. Combining this result with Theorem \ref{th1.1}, we give a positive answer to  Question 1 for entire functions $f$ and $g$.
\end{rem}

\begin{rem}
	Obviously, if $f\equiv g$ is an entire function, Theorem \ref{th1.1} can be reduced to Theorems \ref{thA}, \ref{thB} and \cite[Corollary 1.7]{Liu24}, and improves Theorems \ref{thC}. What's more, our approaches differ significantly from the methods outlined in Theorems \ref{thA}, \ref{thB},  \ref{thC} and \ref{thD}.
\end{rem}

\begin{rem}
	The restriction that non-zero Picard exceptional small functions of $f$ and $g$ can not be removed. For example, let $f(z)=e^z$, $g(z)=e^{2z},$ $L(z,f)=f''+f'$, $L(z,g)=g''+g'$, then $f^3L(z,g)=6e^{5z}$, $g^4L(z,f)=2e^{9z}$ and $0$ is a Picard exceptional value of  $f^3L(z,g)$ and $g^5L(z,f)$.
\end{rem}

\begin{rem}
	\label{rm4}
	If $m=n=1$, then Theorem \ref{th1.1} does not hold. For example, let $f(z)=e^{-p(z)}$, $g(z)=e^{p(z)}$, where $p$ is a non-constant polynomial, then $fg-2=-1$ has no zeros.
\end{rem}

From Theorem D, it is evident that a small function  $\alpha$ respect to $f$ and $g$ can not serve as a common generalized Picard exceptional function for both $f^nL(,g)$ and $g^nL(z,f)$. Naturally, this prompts the question\cite[Question 1.1]{Liu25}: Given any two transcendental meromorphic functions $F(z)$ and $G(z)$, what can we obtain
for their common or different generalized Picard exceptional values or small functions of $F$ and $G$?  The direct consideration of  this question is rather difficult, thus Liu and Liu\cite{{Liu25}} raised the following  question.

\begin{Que}\cite[Question 2.1]{Liu25}
	Let $f$ and $g$ be two transcendental meromorphic functions and $m$, $n$ be positive integers. Do $f^ng$ and $g^mf$ have simultaneously the generalized  Picard exceptional values or small functions?
\end{Que}

Liu and Liu\cite{Liu25} considered Question 2 and obtained partial results.
\begin{renewcommand}{\thethm}{\Alph{thm} }
	\begin{thm}\cite[Theorem 2.1 and Theorem 3.1]{Liu25}
		\label{thE}
		Let f and g be transcendental entire functions.
		\begin{itemize}
			\item [(1)] If $\min\{m,n\}\ge2$, then $f^ng$ and $g^mf$ can not have simultaneously non-zero generalized
			Picard exceptional value.
			\item [(2)]If $n\ge2$  then $f^ng$ and $gf$ can not have simultaneously non-zero generalized
			Picard exceptional value except that $f(z) = se^T(z)$ and $g(z)=s^{-2}a_1e^{-2T(z)}+a_2s^{-1}e^{-T(z)}$, where $T(z)$ is an entire
			function and $s$ is a complex constant. In this case, $a_1$ is the Picard exceptional value of $f^2g$ and $a_2$ is the
			Picard exceptional value of $fg$.
			\item [(3)]  If $\min\{m,n\}\ge3$, $k$ is a positive integer, then $f^ng^{(k)}$ and $g^mf^{(k)}$ can not have simultaneously non-zero
			Picard generalized exceptional small functions.
		\end{itemize}
	\end{thm}
\end{renewcommand}

Liu and Liu asked the following question related Theorem \ref{thE}-$(3)$.
\begin{Que}\cite[Remark 3.2]{Liu25}
	Can we reduce $n\ge3$ to $n\ge2$ in Theorem \ref{thE}-$(3)$?
\end{Que}

 We consider Questions 2 and 3 and obtain the following results.
\begin{thmx}
	\label{th1.2}
	Let $f$ and $g$ be both transcendental entire functions.
	\begin{itemize}
		\item [(1)] If $\min\{m,n\}\ge2$, then $f^nL(z,g)$ and $g^m(z)L(z,f)$ cannot have simultaneously non-zero generalized Picard exceptional small functions, where $L(z,f)$ and  $L(z,g)$ are non-zero linear differential polynomial  with small entire functions of $f$ and $g$ as its coefficients.
		\item [(2)]If $n\ge2$ and  $T(r,g)\le O(T,f)$,  then  $f^ng$ and $gf$  can not have simultaneously non-zero generalized
		Picard exceptional value except that $f(z)=A(z)e^{B(z)}$, $g(z)=A_4(z)e^{-nB(z)}+A_3(z)e^{-B(z)}$, where $A,A_3,A_4$ are small entire function of $f$ and $g$, $B$ is an entire function.  What's more, $A^nA_4$ is the Picard exceptional small  function of $f^ng$, $AA_3$ is the Picard exceptional small function of $fg$.
	\end{itemize}
\end{thmx}

\begin{rem}
	If  $\rho_2(f)<1$ and $\rho_2(g)<1$, $L(z,f)$ and  $L(z,g)$ are linear delay-differential  polynomial  with small entire functions of $f$ and $g$ as its coefficients,  then Theorem \ref{th1.2}-$(1)$ still holds by using the similar proof of Theorem \ref{th1.2}.
\end{rem}
\begin{rem}
	Let $L(z,f)=f^{(k)}$ and $L(z,g)=g^{(k)}$, then Theorem \ref{th1.2}-(1)  gives a  positive answer to Question 3. If $k$=0, Theorem \ref{th1.2}-$(1)$ improves Theorem \ref{thE}-$(1)$ to the case of Picard exceptional small functions.
\end{rem}

\begin{rem}
	Note that $f(z)=e^{z}$ and $g(z)=2e^{-3z}+e^{-z}$  satisfy the relationships $f^3(z)g(z)=2+e^{2z}$ and $g(z)f(z)=2e^{-2z}+1$.  Consequently, $2$ is a Picard exceptional value of $f^3g$, $1$  is a Picard exceptional value of $gf$. However, $g(z)=2e^{-3z}+e^{-z}$  does not conform to the form $g(z)=s^{-2}a_1e^{-2T(z)}+a_2s^{-1}e^{-T(z)}$ specified in Theorem \ref{thE}-$(2)$. This suggests that there may be a  gap in the proof of Theorem  \ref{thE}-$(2)$.
	Our  Theorem \ref{th1.2}-$(2)$ rectify the gap using different methods from those in Theorem  \ref{thE} under the condition $T(r,g)\le O(T(r,f))$.  
	What's more, the condition $T(r,g)\le O(T(r,f))$ is  necessary. For example, let  $g(z)=e^{z^3}+e^{-3z}$ and $f(z)=e^{z}$,  which satisfies $f^2g=e^{z^3+2z}+e^{-z}$, $gf=e^{z^3+z}+e^{-2z}$ and $T(r,g)>O(T(r,f))$. Here $e^{-z}$  is a Picard exceptional small function of $f^2g$, $e^{-2z}$ is a Picard exceptional small function of $gf$. But $g(z)$  does not conform to the form specified in Theorem \ref{th1.2}-$(2)$. 

\end{rem}
\begin{rem}
	\label{rm7}
	
	An example of Theorem \ref{th1.2}-$(2)$ in the context of small functions is given as follows: $f(z)=e^{z^2+z}e^{z^5}$, $g(z)=e^{-5z^5}+e^{z^3}e^{-z^5}$,  which satisfies
	$f^5g=e^{5z^2+5z}+e^{4
		z^5}e^{5z^2+5z+z^3}$, $fg=e^{z^2+z}e^{-4z^5}+e^{z^2+z+z^3}$. Here, $e^{5z^2+5z}$ is a Picard exceptional small function of $f^5g$, $e^{z^2+z+z^3}$ is a Picard exceptional small function of $fg$.
\end{rem}
Remark \ref{rm4}   can also be used to demonstrate that Theorem \ref{th1.2} does not hold when $n=m=1$.  Combining this with Theorem \ref{th1.2} and  Remark \ref{rm7}, we obtain the following corollary, which completely solves Question 2 for the case where $f$ and $g$ are entire functions.
\begin{cor}
	Let $f$ and $g$ be two transcendental entire functions and $m$, $n$ be positive integers, then $f^ng$ and $g^mf$ can not have simultaneously the generalized  Picard exceptional small functions except $\min\{m,n\}=1$.
\end{cor}

\section{Proof  of Theorems  \ref{th1.1} and \ref{th1.2}}
\begin{proof}[Proof of Theorem \ref{th1.1}]Let 
	$\underset{r\rightarrow \infty}{\lim\sup}\frac{T(r,f)}{T(r,g)}=k$. If $k\not=\infty$, then $T(r,f)\le O(T(r,g))$, if $k=\infty$, then $T(r,g)\le O(T(r,f))$.
	Therefor, without loss of generality, we will prove Theorem \ref{th1.1} for the case $T(r,g)\le O(T(r,f))$.
	
			 We claim $f^nL(z,g)-\alpha$ or $g^nL(z,f)-\alpha$ has infinitely many zeros, where $\alpha$ is a small function of $f$ and $g$. Otherwise,  $f^nL(z,g)-\alpha$  and $g^mL(z,f)-\alpha$ has finitely many zeros, by the  Weierstrass's factorization theorem\cite[pp. 145]{stein}, we have
			\begin{align}
				\label{0.1}
				f^n(z)L(z,g)-\alpha(z)=p(z)e^{b(z)},
			\end{align}
			where $p$  is  a  small meromorphic function of $f$ and $g$, $b$ is an entire function, and
			\begin{align}
				\label{001}
				g^m(z)L(z,f)-\alpha(z)=h(z)e^{d(z)},
			\end{align}
			where $h$  is  a  small meromorphic function of $f$ and $g$, $d$ is an entire function.

			Since $T(r,g)\le O(T(r,f))$, then from \eqref{0.1}, we get $T(r,e^b)\le O(T(r,f))$. Which means $b$ is a small function of $f$. 
			Differentiating the equation \eqref{0.1}, we get
			\begin{align}
				\label{0.2}
				nf^{n-1}f'L(z,g)+f^nL'(z,g)-\alpha'=p_1e^b,
			\end{align}
			where $p_1=p'+b'p$ and $p_1$ is a small function of $f$.  

			\setcounter{case}{0}
	\begin{case}
		\rm{$p_1=p'+b'p\equiv0$.
		Since  $p_1=p'+b'p\equiv0$, by  integrating, we can obtain $p(z)=c_1e^{-b(z)}$, where $c_1$ is a non-zero constant. 
			Substituting $p$ into \eqref{0.1}, we get $f^n(z)L(z,g)=\alpha+c_1$.  
			Since $\alpha+c_1$ is a small function of $f$ and $L(z,g)$ is an entire function, then we can see $N(r,\frac{1}{f})=S(r,f)$.
 By  $f^nL(z,g)=\alpha+c_1$
and  the logarithmic derivative lemma \cite[Theorem 2.2]{hayman}, then 
\begin{align}
	\label{003}
	nm(r,f)=m(r,\frac{\alpha+c_1}{L(z,g)})
	&\le T(r,L(z,g))+S(r,f) \\ \nonumber
	&\le m(r,g)+m(r,\frac{L(z,g)}{g})+S(r,f) \\ \nonumber
	&\le m(r,g)+S(r,f).
\end{align}
Thus, $T(r,f)=O(T(r,g))$. Therefor $S(r,f)=S(r,g)$.
Then from \eqref{001}, we get $T(r,e^d)\le O(T(r,g))$. Which means $d$ is a small function of $f$ and $g$. 

Differentiating the equation \eqref{001}, we get
			\begin{align}
				\label{002}
				ng^{m-1}g'L(z,f)+g^mL'(z,f)-\alpha'=h_1e^{d},
			\end{align}
			where $h_1=h'+d'h$ and $h_1$ is a small function of $g$.  

\setcounter{subsection}{1}
	\setcounter{subcase}{0}
	\renewcommand{\thesubcase}{\arabic{subsection}.\arabic{subcase}}

	\begin{subcase}$h_1\equiv0$. 
		\rm{
			Since $h_1\equiv0$, then as  the same as $p_1\equiv0$, we get $g^mL(z,f)=\alpha+c_2$, where $c_2$ is a non-zero constant. Therefor $N(r,\frac{1}{g})=S(r,g)$. 
Since  $g^mL(z,f)=\alpha+c_2$,
by the logarithmic derivative lemma \cite[Theorem 2.2]{hayman}, then 

\begin{align}
	\label{004}
	mm(r,g)\le m(r,f)+S(r,f)
\end{align}
From \eqref{003} and \eqref{004}, we get 
$nm(r,f)\le m(r,g)+S(r,f)\le \frac{1}{m}m(r,f)+S(r,g)$, which means $n<1$, that is impossible.

		}
	\end{subcase}

\begin{subcase}
	\rm{
	 $h_1\not\equiv0$. By eliminating $e^d$ from equations \eqref{001} and \eqref{002}, we can obtain
			\begin{align}
				\label{005}
				h_1g^mL(z,f)-hng^{m-1}g'L(z,f)-hg^mL'(z,f)=h_1\alpha-h\alpha'=h_2.
			\end{align}

If $h_2\equiv0$,  by integrating, then $h=\alpha c_3e^{-d}$, where $c_3$ is a non-zero constant. Substituting $h$ into \eqref{001}, we get $g^mL(z,f)=(c_3+1)\alpha$. Then  using the same methods as Subcase 1.1, we can get a contradiction. 

If $h_2\not\equiv0$, 
suppose $N(r,\frac{1}{g})\not=S(r,g)$, since coefficients of the equation \eqref{005} are small functions of $g$,   then there exists a zero $z_0$ of $g$ such that the coefficients of the equation \eqref{005} are neither zero nor infinite at that point.  Substituting $z_0$ into the equation \eqref{005}, we easily obtain a contradiction.Therefore, $N(r,\frac{1}{g})=S(r,f)$. By \eqref{005},then 
\begin{align}
	\label{006}
	m(r,\frac{1}{g^m})&\le m(r,\frac{1}{h_2})+m(r,h_1L(z,f)-hm\frac{g'}{g}L(z,f)-hL'(z,f)) \\ \nonumber
	&\le m(r,f)+S(r,f). 
\end{align}
By $N(r,\frac{1}{g})=S(r,f)$ and \eqref{006}, then 
$mT(r,g)=mm(r,\frac{1}{g})+mN(r,\frac{1}{g})+O(1)\le T(r,f)+S(r,f)$. Combining this and $\eqref{003}$, we get $nm(r,f)\le \frac{1}{m}m(r,f)+S(r,f)$, which is impossible, since $n\ge2$. 
	}
\end{subcase}

\begin{case}
	\rm{
	$p_1=p'+b'p\not\equiv0$.  By eliminating $e^b$ from equations \eqref{0.1} and \eqref{0.2}, we can obtain
			\begin{align}
				\label{1.3}
				p_1f^nL(z,g)-pnf^{n-1}f'L(z,g)-pf^nL'(z,g)=p_1\alpha-p\alpha'=p_2.
			\end{align}

			 If $p_2=p_1\alpha-p\alpha'\equiv0$, by integrating,   we can obtain $\frac{\alpha}{p}=c_2e^b$. Substituting $p$ into \eqref{0.1}, we get $f^n(z)L(z,g)=(1+c_2)\alpha$.
			 This situation is the same as Case 1, we omit the proof here. 

			 If $p_2\not\equiv0$, by  using the  same method as in case $h_2\not\equiv0$, then   $N(r,\frac{1}{f})=S(r,f)$, $m(r,\frac{1}{f})=T(r,f)+S(r,f)$. 
			 From \eqref{1.3},  as the same as \eqref{006} we can get 
			 \begin{align}
				\label{008}
				nT(r,f)+S(r,f)=m(r,\frac{1}{f^n})\le T(r,g) +S(r,f),
			 \end{align}
 which means $T(r,f)=O(T(r,g))$, $S(r,f)=S(r,g)$. Then from \eqref{001}, we get $T(r,e^d)\le O(T(r,g))$. Which means $d$ is a small function of $f$ and $g$.  
 Differentiating the equation \eqref{001}, we also get \eqref{002}. 
\setcounter{subsection}{2}
	\setcounter{subcase}{0}
	\renewcommand{\thesubcase}{\arabic{subsection}.\arabic{subcase}}

	\begin{subcase}
		$h_1\equiv0$. 
			Since $h_1\equiv0$, then as  the same as Subcase 1.1, we get $g^mL(z,f)=\alpha+c_2$, where $c_2$ is a non-zero constant. Therefor $N(r,\frac{1}{g})=S(r,g)$. 
Since  $g^mL(z,f)=\alpha+c_2$, then $m(r,g)\le \frac{1}{m}m(r,f)$. By this and \eqref{008}, we get  $nT(r,f)\le \frac{1}{m}m(r,f)$. Which is impossible, since $n\ge2$. 
	\end{subcase}

	\begin{subcase}
		$h_1\not\equiv0$. By eliminating $e^d$ from equations \eqref{001} and \eqref{002}, we can obtain \eqref{005}. 
	\end{subcase}
If $h_2\equiv0$,  Then  using the same methods as Subcase 1.1, we can get a contradiction. 

If $h_2\not\equiv0$, from \eqref{005}, we get the zero of $g$ must be the zero of $h_2$. Since $h_2$ is a small function of $g$, then $N(r,\frac{1}{g})=S(r,f)$. By \eqref{005},then we can also get \eqref{006}. Combining \eqref{006} and \eqref{008}, we get $n<1$, which is impossible. 
	}
\end{case}

		}

	\end{case}

\end{proof}

\begin{proof}[Proof of Theorem \ref{th1.2}]  
Theorem \ref{th1.2} will be proved in two cases below.

	\setcounter{case}{0}
\begin{case}\rm{
	 $\min\{m,n\}\ge2$. 
Let 
	$\underset{r\rightarrow \infty}{\lim\sup}\frac{T(r,f)}{T(r,g)}=k$. If $k\not=\infty$, then $T(r,f)\le O(T(r,g))$, if $k=\infty$, then $T(r,g)\le O(T(r,f))$.
	Therefor, without loss of generality, we will prove Theorem \ref{th1.2}-(1) for the case $T(r,g)\le O(T(r,f))$.	
	Suppose
	 $f^nL(z,g)$ and $g^mL(z,f)$ have simultaneously non-zero generalized  Picard exceptional small functions. By Weierstrass's factorization theorem, we assume that
			\begin{equation}
				\label{0.8}
				\begin{cases}
					f^nL(z,g)-a_1(z)=p_1(z)e^{b_1(z)},\\
					g^mL(z,f)-a_2(z)=p_2(z)e^{b_2(z)},
				\end{cases}
			\end{equation}
			where $a_1$ and $b_1$ are non-zero small functions of $f^nL(z,g)$, $a_2$ and $b_2$ are non-zero small functions of $g^mL(z,f)$.	
			Since $a_1$ is a small function of $f^nL(z,g)$, then $a_1$ is a small function of $f$. Consequently, the subsequent proof follows a similar approach to that of Theorem  \ref{th1.1} and  we can get a contradiction, we omit the proof here. Now  the proof of Theorem \ref{th1.2}-(1) is completed.
}
\end{case}

\begin{case}
	\rm{  $n\ge2$ and $T(r,g)\le O(T(r,f))$.
Suppose $f^ng$ and $gf$ have simultaneously non-zero  generalized Picard exceptional small functions. By Weierstrass's factorization theorem, we assume that
					\begin{equation}
						\label{0.9}
						\begin{cases}
							f^ng-a_1(z)=p_1(z)e^{b_1(z)},\\
							gf-a_2(z)=p_2(z)e^{b_2(z)},
						\end{cases}
					\end{equation}
					where $a_1$ and $b_1$ are non-zero small functions of $f^ng$, $a_2$ and $b_2$ are non-zero small functions of $gf$.
Since $T(r,g)\le O(T(r,f))$, then $a_i$ and $b_i$ ($i=1,2$)  are small function of $f$. 

Differentiating the first equation  of \eqref{0.9}, we get
			\begin{align}
				\label{0.10}
				nf^{n-1}f'g+f^ng'-a_1'=h_1e^{b_1},
			\end{align}
			where $h_1=p_1'+b_1'p_1$ and $h_1$ is a small function of $f$.  
\setcounter{subsection}{2}
	\setcounter{subcase}{0}
	\renewcommand{\thesubcase}{\arabic{subsection}.\arabic{subcase}}

	\begin{subcase}
\rm{
		$h_1=p_1'+b_1'p_1\equiv0$. Since  $h_1\equiv0$, by  integrating, we can obtain $p_1(z)=c_1e^{-b_1(z)}$, where $c_1$ is a non-zero constant. 
			Substituting $p_1$ into \eqref{0.1}, we get $f^n(z)g=a_1+c_1$.  Since $a_1+c_1$ is a small function of $f$ and $g$ is an entire function, then we can see $N(r,\frac{1}{f})=S(r,f)$ and $nT(r,f)\le m(r,g)+S(r,f)$. Therefor $S(r,f)=S(r,g)$, $a_i$ and $b_i$ ($i=1,2$)  are small function of $f$ and $g$.

By  the  Weierstrass's factorization theorem, we can get $f(z)=A(z)e^{B(z)}$, where $A$ is a small function of $f$, $B$ is an entire function.
					Substituting the expression for $f$ into 
					the second equation of \eqref{0.9}, we get
					\begin{align}
						\label{1.0}
						g(z)=A_2(z)e^{b_2(z)-B(z)}+A_3(z)e^{-B(z)},
					\end{align}
	where $A_2$, $A_3$ are small functions of $f$ and $g$. Substituting the expressions for $f$ and $g$ into $f^n(z)g=a_1+c_1$, then 

	\begin{align}
	\label{3.1}
	\alpha_1e^{(n-1)B+b_2}+\alpha_2e^{(n-1)B}=a_1+c_1,
	\end{align}
	where $\alpha_i,i=1,2$ are non-zero small functions of $f$ and $g$. 
	Since $\alpha_i (i=1,2), a_1+c_1$  are small functions of $f=Ae^B$ and $n>1$, then they are also the small function of $e^{(n-1)B}$. By \eqref{3.1} and using the second fundamental theorem concerning three small functions to $e^{(n-1)B}$, we easily can get a contradiction. 

}
	\end{subcase}

\begin{subcase}
$h_1=p_1'+b_1'p_1\not\equiv0$.  By eliminating $e^{b_1}$ from equations \eqref{0.9} and \eqref{0.10}, we can obtain
			\begin{align}
				\label{01.3}
				h_1f^ng-p_1nf^{n-1}f'g-p_1f^ng'=h_1a_1-p_1a_1'=h_2.
			\end{align}

			 If $h_2=h_1a_1-p_1a_1'\equiv0$, by integrating,   we can obtain $\frac{a_1}{p_1}=c_2e^{b_1}$, where $c_2$ is a non-zero constant. Substituting $p_1$ into \eqref{0.9}, we get $f^ng=(1+c_2)a_1$.
			 This situation is the same as Subcase 2.1, we omit the proof here. 

			 If $h_2\not\equiv0$, suppose $N(r,\frac{1}{f})\not=S(r,f)$, since coefficients of the equation \eqref{01.3} are small functions of $f$,   then there exists a zero $z_0$ of $f$ such that the coefficients of the equation \eqref{01.3} are neither zero nor infinite at that point.  Substituting $z_0$ into the equation \eqref{01.3}, we easily obtain a contradiction.
			Therefore, $N(r,\frac{1}{f})=S(r,f)$,  then  $m(r,\frac{1}{f})=T(r,f)+S(r,f)$. 
From \eqref{01.3},  we can get 
			 \begin{align*}
				nT(r,f)+S(r,f)=m(r,\frac{1}{f^n})\le T(r,g) +S(r,f),
			 \end{align*}
			 Therefor $S(r,f)=S(r,g)$, $a_i$ and $b_i$ ($i=1,2$)  are small function of $f$ and $g$.  By  the  Weierstrass's factorization theorem, we can get $f(z)=A(z)e^{B(z)}$, where $A$ is a small function of $f$, $B$ is an entire function. 
			 

			 Substituting the expression for $f$ into 
					the second equation of \eqref{0.9},  we also get \eqref{1.0}. Substituting $f$ and \eqref{1.0} into the first equation of \eqref{0.9}, we get
					\begin{align}
						\label{1.1}
						a_3e^{(n-1)B+b_2}+a_4e^{(n-1)B}-a_1=p_1e^{b_1},
					\end{align}
				where $a_3,a_4$ are non-zero small functions of $f$ and $g$.
				Since $a_1,a_3,a_4, p_1$ are small functions of $f$ and $a_1$ is a small function of $e^{b_1}$, by \cite[Theorem 1.56]{ccy2003} and \eqref{1.1}, 
				we  get  $a_3e^{(n-1)B+b_2}=a_1$. Substituting the
			expressions for $e^{b_2}$ into \eqref{1.0}, we get $g(z)=A_4(z)e^{-nB(z)}+A_3(z)e^{-B(z)}$, where $A_4 $ is a non-zero small function of $g$. Substituting $f$ and $g$ in \eqref{0.9} and using the second
fundamental theorem concerning three small functions, we get $A^nA_4=a_1$, $AA_3=a_2$.  Now we get the Theorem \ref{th1.2}-(2). 		 
			  
\end{subcase}

	}
\end{case}

\end{proof}


\end{document}